\documentclass[11pt,twoside]{amsart}
\usepackage{amssymb,amsthm,amsmath,amstext}
\usepackage{mathrsfs}  % allows nice script font for sheaves
\usepackage{bm}        % allows bold italic letters in math mode
\usepackage{mathtools} % allows more extendible arrows
\usepackage{graphicx}
\usepackage[all]{xy}
\usepackage{color}
\usepackage{enumitem}

\usepackage[left=1.5in,top=1.5in,right=1.5in,bottom=1.55in]{geometry}

\usepackage{amscd}
\usepackage[latin1]{inputenc}
\usepackage[OT1]{fontenc}
\usepackage[bbgreekl]{mathbbol}

\theoremstyle{plain}
\newtheorem{theorem}{Theorem}[section]

\newtheorem{proposition}[theorem]{Proposition}
\newtheorem{lemma}[theorem]{Lemma}

\newtheorem{notation}[theorem]{Notation}

\theoremstyle{definition}
\newtheorem{definition}[theorem]{Definition}
\newtheorem{example}[theorem]{Example}

\newtheorem{remark}[theorem]{Remark}

\newcommand{\sheaf}[1]{\mathscr{#1}}
\newcommand{\KK}{\sheaf{K}}
\newcommand{\LL}{\sheaf{L}}
\newcommand{\OO}{\sheaf{O}}

\newcommand{\EE}{\sheaf{E}}
\newcommand{\FF}{\sheaf{F}}
\newcommand{\GG}{\sheaf{G}}
\newcommand{\HH}{\sheaf{H}}
\newcommand{\NN}{\sheaf{N}}
\newcommand{\VV}{\sheaf{V}}
\newcommand{\WW}{\sheaf{W}}

\newcommand{\PP}{\sheaf{P}}
\newcommand{\II}{\sheaf{I}}

\newcommand{\TT}{\sheaf{T}}
\newcommand{\CC}{\sheaf{C}}

\newcommand{\A}{\mathbb A}
\newcommand{\C}{\mathbb C}
\renewcommand{\P}{\mathbb P}
\newcommand{\Q}{\mathbb Q}
\newcommand{\N}{\mathbb{N}}
\newcommand{\F}{\mathbb{F}}

\DeclareMathOperator{\gin}{\mathrm{gin}}
\DeclareMathOperator{\initialTerm}{\mathrm{in}}

\usepackage{hyperref}
\hypersetup{pdftitle={Degenerations of Gushel-Mukai fourfolds}}
\hypersetup{pdfauthor={Christian Boehning, and Hans-Christian Graf von Bothmer}}
\hypersetup{colorlinks=true,linkcolor=blue,anchorcolor=blue,citecolor=blue}

\begin{document}

\title[Degenerations of Gushel-Mukai fourfolds]{Degenerations of Gushel-Mukai fourfolds, with a view towards irrationality proofs}

%\author[Auel]{Asher Auel}
%\address{Asher Auel, Department of Mathematics\\
%Yale University\\
%10 Hillhouse Avenue\\
%New Haven, CT 06511, USA}
%\email{auel@yale.edu}

\author[B\"ohning]{Christian B\"ohning}
\address{Christian B\"ohning, Mathematics Institute, University of Warwick\\
Coventry CV4 7AL, England}
\email{C.Boehning@warwick.ac.uk}

\author[Bothmer]{Hans-Christian Graf von Bothmer}
\address{Hans-Christian Graf von Bothmer, Fachbereich Mathematik der Universit\"at Hamburg\\
Bundesstra\ss e 55\\
20146 Hamburg, Germany}
\email{hans.christian.v.bothmer@uni-hamburg.de}

%\thanks{$^1$ Supported by Heisenberg-Stipendium BO 3699/1-2 of the DFG (German Research Foundation) during the initial stages of this work}
%\thanks{$^2$ Partially supported by the RTG 1670 of the  DFG (German Research Foundation)}

\begin{abstract}
We study a certain class of degenerations of Gushel-Mukai fourfolds as conic bundles, which we call tame degenerations and which are natural if one wants to prove that very general Gushel-Mukai fourfolds are irrational using the degeneration method due to Voisin, Colliot-Th\'{e}l\`{e}ne-Pirutka, Totaro et al. However, we prove that no such tame degenerations do exist. 
\end{abstract}

\maketitle

\section{Introduction}\label{sIntroduction}

We work over the complex numbers $\C$ unless otherwise stated.

\begin{definition}\label{dGushelMukai}
We fix a five-dimensional vector space $W$. A \emph{Gushel-Mukai fourfold} (\emph{GM fourfold} for short) $X$ is a smooth dimensionally transverse 
intersection
\[
X = Q \cap \mathrm{Gr}(2, W) \cap H
\]
of the Grassmannian $\mathrm{Gr}(2, W) \subset \P (\Lambda^2 W)$, a hyperplane $H$ and a quadric $Q$ in $\P (\Lambda^2 W)$.
\end{definition}

It is a very interesting open problem to decide whether a very general GM fourfold $X$ is rational or not. Part of the interest of this question comes from the conjectural similarity of the picture for GM fourfolds to the picture for cubic fourfolds. Also, GM fourfolds have recently been studied from various perspectives, geometric, Hodge theoretic and derived categorical \cite{DK15-1}, \cite{DK16-1} \cite{KP16}. One may ask whether the very general GM fourfold is even not stably rational, and then one can seek to apply the degeneration method of Voisin, Colliot-Th\'{e}l\`{e}ne-Pirutka, Totaro et al. \cite{Voi15}, \cite{CT-P16}, \cite{To16} that has led to such a multitude of applications recently. In fact, GM fourfolds are birational to a certain class of conic bundles over $\P^3$ with sextic discriminant surfaces, see Proposition \ref{pGMasConicBundle}. This has been known for a while and can be extracted from the work by Debarre and Kuznetsov cited before. Now the main theorem, Theorem 2.6, of \cite{ABBP16}, has the following direct consequence when combined with the specialization principle in \cite{CT-P16}. This special case of \cite[Thm. 2.6]{ABBP16} is all that is relevant for us here.

\begin{theorem}\label{tDegConicBundles}
Let $\EE_t$ be a family of rank $4$ vector bundles over the base variety $\P^3$, parametrized by $t\in \A^1$, in other words, a vector bundle $\EE$ on $\P^3\times \A^1$. Let $\Phi \colon \EE \to \EE^{\vee} \otimes \LL$, for $\LL$ some line bundle on $\P^3\times \A^1$, be a morphism, symmetric up to twist by the line bundle, such that its degeneracy locus $\CC \subset \P (\EE ) \to \P^3\times \A^1$ defines a flat family of conic bundles over $\P^3$, or, equivalently, that the rank of the quadratic forms defined by $\Phi$ on the fibers of $\EE$ never drops to zero. We denote by $\CC_t$ the conic bundle over $\P^3$ in $\P (\EE_t)$ corresponding to $t\in \A^1$. Suppose that for general $t$, $\CC_t$ is smooth. Let $\Delta$ be the discriminant of $\CC_0$, the possibly reducible surface in $\P^3$ above points of which the fibers of $\CC_0$ are singular conics. Suppose that the following are true:
\begin{enumerate}
\item
$\Delta$ breaks up into two irreducible components $\Delta_1$ and $\Delta_2$, and the fibers of $\CC_0$ over general points in $\Delta_1$ and $\Delta_2$ consist of two distinct lines (not a double line).
\item
The double covers $\tilde{\Delta}_1 \to \Delta_1$ and $\tilde{\Delta}_2 \to \Delta_2$ determined by $\CC_0$ are irreducible, or, put differently, do not split.
\item
Let $C_j$ be the irreducible components of the intersection curve $\Delta_1\cap \Delta_2$. Then the fibers of $\CC_0$ over a general point in each $C_j$ are two distinct lines, and the corresponding double covers $\tilde{C}_j\to C_j$ split.
\item
$\Delta_1$ and $\Delta_2$ are smooth along every $C_j$.
\item
The total space $\CC_0$ is only mildly singular in the sense that it should have a Chow universally trivial resolution $\varpi \colon \tilde{\CC}_0 \to \CC_0$, which means that for any overfield $K$ of $\C$, the pushforward $\varpi_*$ gives an isomorphism of Chow groups between $\mathrm{CH}_0((\tilde{\CC}_0)_K)$ and $\mathrm{CH}_0 ((\CC_0)_K)$.
\end{enumerate}
Then $\tilde{\CC}_0$ has a nontrivial unramified Brauer group, and for very general $t\in \A^1$, $\CC_t$ is not stably rational.
\end{theorem}

The first problem we consider in this article is if this theorem is applicable to the conic bundles arising from GM fourfolds, and our answer will be negative if one restricts attention to a rather natural class of ``tame degenerations". 

Since explaining the naturality of this class of degenerations from the point of view of Theorem \ref{tDegConicBundles} requires setting up more notation and a somewhat more detailed analysis of the class of conic bundles arising from GM fourfolds, we cannot do this in this Introduction but have to postpone it to Section \ref{sGMConicBundles}, Definitions \ref{dAdmissibleDegeneration2}, \ref{dTameDegeneration2} and the discussion following Remark \ref{rDiscriminant}. Suffice it to say that the main geometric tool in analysing tame degenerations is the fact that the discriminant of such a conic bundle is contact to a Kummer surface $K$. These are $16$ nodal quartic surfaces in $\P^3$. 

A wider class of degenerations can be obtained by letting such a $K$ degenerate to a null-correlation quartic with worse singularities, see Definition \ref{dKummerETC}. We should point out that we carried out a computer search for such more general degenerations involving null-correlation quartics with worse singularities, but did not find any examples where Theorem \ref{tDegConicBundles} is applicable, either. Another approach could of course be to consider other types of degenerations of GM fourfolds that are not related to their birational models as conic bundles, but to some other geometric structure they might carry.

\medskip

\textbf{Acknowledgments.} We would like to thank Asher Auel, Alexander Kuznetsov, Alena Pirutka, Kristian Ranestad for useful discussions and for drawing our attention to this interesting problem. 

It preparation for this article we did extensive computer experiments using the computer algebra system {\tt Macaulay2} \cite{M2} and Jakob Kr\"oker's  {\tt Macaulay2}-packages {\tt BlackBoxIdeals} and {\tt FiniteFieldExperiments} \cite{Kroe}.

\section{Gushel-Mukai fourfolds as conic bundles over $\P^3$ and tame degenerations}\label{sGMConicBundles}

Gushel-Mukai fourfolds as in Definition \ref{dGushelMukai} are birational to conic bundles over $\P^3$ of a certain type which we now define.

\begin{definition}\label{dBundles}
Let $V$ be a four-dimensional vector space, and consider the bundle $\Omega^1 (2)=\Omega^1_{\P^3}(2)$ on $\P^3 = \P (V)$. A \emph{Gushel-Mukai vector bundle} (\emph{GM vector bundle}) $\EE_{\sigma}$ is the cokernel of a nowhere vanishing section 
\[
\sigma \colon \OO_{\P^3} \to \OO_{\P^3} (1) \oplus \Omega^1_{\P^3} (2).
\]
A \emph{null-correlation} bundle $\NN_{\bar{\sigma}}$ is the cokernel of a nowhere vanishing section 
\[
\bar{\sigma} \colon \OO_{\P^3} \to \Omega^1_{\P^3} (2).
\]
\end{definition}

\begin{proposition}\label{pGMasConicBundle}
A general GM fourfold $X$ is birational to a conic bundle 
\[
\xymatrix{
\CC_{\varphi, \sigma} \ar@{^{(}->}[r]\ar[rd] & \P (\EE_{\sigma}^{\vee})\ar[d]\\
  & \P^3
}
\]
associated to  a symmetric map
\[
\varphi \colon \EE_{\sigma}^{\vee } \to \EE_{\sigma}
\]
for some GM vector bundle $\EE_{\sigma}$. 
\end{proposition}

\begin{proof}
Place yourself in the set-up of Definition \ref{dGushelMukai}. 
The basic idea is to look at the incidence correspondence ``point on line"
\[
\xymatrix{
 & Z \ar@{^{(}->}[r]\ar[ld]^{p_1} \ar[rd]^{p_2} & \P (W) \times \mathrm{Gr}(2, W)\subset \P(W) \times \P (\Lambda^2 W) \\
 \P (W) &  & \mathrm{Gr}(2, W) 
}
\]
Via $p_1$, $Z$ becomes isomorphic to $\P (\TT_{\P(W)}) \simeq \P (\TT_{\P (W)} (-2))$. The hyperplane $H$ in the definition of a GM fourfold corresponds to a linear form (well-defined up to scalars), $h \in \Lambda^2 W^{\vee}$, and similarly the quadric $Q$ in the definition of GM fourfold is equivalent to the datum of a quadratic form (up to scalars) $q\in \mathrm{Sym}^2 \Lambda^2 W^{\vee}$. Since 
\[
H^0 (\P (W), \Omega^1_{\P(W)}(2))\simeq \Lambda^2 W^{\vee}
\]
a generic $h$ determines a section 
\[
\sigma_h \colon \OO_{\P(W)} \to \Omega^1_{\P(W)}(2)
\]
vanishing in a single point $P$ in $\P (W)=\P^4$, 
and a corresponding quotient sheaf $\VV_{\sigma_h}$, which is a bundle outside of that point $P$. We will now work on $\P (W)^0:= \P (W) -\{ P \}$. Then, we have an honest bundle $\VV_{\sigma_h}^0$ there, and dually, $(\VV_{\sigma_h}^0)^{\vee} \hookrightarrow \TT_{\P(W)^0}(-2)$ a subbundle. We also have a corresponding diagram
\[
\xymatrix{
 & Z^0 \ar@{^{(}->}[r]\ar[ld]^{p_1} \ar[rd]^{p_2} & \P (W)^0 \times \mathrm{Gr}(2, W)\subset \P(W)^0 \times \P (\Lambda^2 W) \\
 \P (W)^0 &  & \mathrm{Gr}(2, W) 
}
\]

The variety $Z(h)^0:=Z^0 \cap p_2^{-1}(\{ h=0\})$ is nothing but $\P ((\VV^0_{\sigma_h})^{\vee})$ via $p_1$. This is now a $\P^2$-bundle over $\P (W)^0 \simeq \P^4-\{ P\}$ via $p_1$. Now since
\[
H^0 (\P (W), \mathrm{Sym}^2 \Omega^1_{\P(W)}(2)) = \mathrm{Sym}^2 \Lambda^2 W^{\vee}
\]
the element $q$ gives a symmetric morphism of bundles
\[
\xymatrix{
(\Omega^1(2)_{\P(W)})^{\vee } \ar[r]^{q} & \Omega^1_{\P(W)}(2) \ar@{->>}[d] \\
\VV^{\vee}_{\sigma_h} \ar[u]\ar[r]^{\bar{q}}  &  \VV_{\sigma_h}
}
\]
Now if we consider $Z(h, q)^0:= Z(h)^0 \cap p_2^{-1}\{ q=0 \}$ we find that this is nothing but the conic bundle
\[
\xymatrix{
\{ q=0\}=\CC_{q} \ar@{^{(}->}[r]\ar[rd]^{p_1} & \P ((\VV^0_{\sigma_h})^{\vee} )\ar[d]^{p_1}\\
 & \P (W)^0
}
\]
Here we a priori allow the possibility that some fibers of $\CC_q$ may be entire $\P^2$'s, in other words, the whole fiber of $\P ((\VV^0_{\sigma_h})^{\vee} )$ at that point, but we will see that that does not happen for generic choices. However, the basic observation is now that via $p_2$ we see that 
\[
\xymatrix{
\CC_{q}\ar[d]^{p_2}\\
\mathrm{Gr}(2, W)\cap Q\cap H =: X
}
\]
is birational to the restriction to $X$ of the projectivization of the universal subbundle on $\mathrm{Gr}(2, W)$ (only birational since we removed $P$). However, this is just birationally a (Zariski-locally trivial) $\P^1$-bundle over $X$. Any section in it will hence be birational to $X$. Now a way of producing such a section is simply by choosing a generic four-dimensional subspace $V \subset W$ in the five-dimensional vector space $W$, and hence a $\P^3=\P(V) \subset \P (W)$ \emph{not containing $P$} (this is how we eventually get rid of the annoying point $P$ which we had to exclude up to now). Indeed, $p_1^{-1}(\P (V))\subset \CC_{q}$ will then be a rational section of $p_2$ because the fiber inside $\CC_q$ over a point in $X$ of $p_2$ is just the line in $\P (W)$ corresponding to it, and intersecting that line with $\P (V)\simeq \P^3$ (generically) picks a point on that line. Hence, if we simply restrict $\CC_q$ to a generic hyperplane $\P (V)\subset \P (W)$:
\[
\xymatrix{
\CC_{\P^3} := \CC_{q}\mid_{\P (V)}\ar[d]^{p_1}\\
\P (V) \simeq \P^3
}
\]
we get a conic bundle birational to $X$. Moreover, it can be checked that for general $\P (V)\subset \P (W)$ and general $H, Q$, this is indeed a ``true" conic bundle, by which we mean that it is a flat projective surjective morphism all of whose fibers are isomorphic to plane conics with general fibers smooth, as follows:  we know that the incidence correspondence $Z$ is a $\P^3$ bundle over $\P (W)$. Using Macaulay2 \cite{BB17-M2} one can show that there are quadrics $Q \subset \P (\Lambda^2 W)$ such that $Z \cap p_2^{-1}(Q\cap \mathrm{Gr}(2, W))$ is a quadric fibration over $\P (W)$ with the properties: 
\begin{enumerate}
\item
The rank of the quadrics is $\le 3$ on a degree $6$ hypersurface; 
\item
it is $\le 2$ on a curve of degree $40$; 
 \item
it is $\le 1$ nowhere 
\end{enumerate}
(indeed, this is the generic behavior). Hence, if we intersect each fiber with a hyperplane $H \subset \P (\Lambda^2 (W))$, the types of intersection behavior we have to avoid to get a flat conic bundle are:
\begin{enumerate}
\item
$H$ contains an entire $\P^3$ fiber.
\item
$H$ contains a $\P^2$ in fibers where the quadrics have rank $2$. 
\end{enumerate}
Now, counting dimensions, the $H$'s containing a given $\P^2$ are codimension $3$ in their parameter space, whence the $H$'s containing a $\P^2$ in some fiber where the quadric has rank $2$ is at least codimension $2$ in $\P (\Lambda^2 W^{\vee})$ for general choice of $Q$ by the above calculation. On the other hand, the codimension of the $H$'s containing an entire $\P^3$ fiber is $-4 + 4=0$, and we know this happens exactly over the one point $P$ for a general choice. Hence choosing $H$ general, and then choosing a $\P (V) \subset \P (W)$ that avoids the point $P$, we get a flat conic bundle.

\medskip

To conclude the proof it remains to remark that 
\[
\Omega^1_{\P(W)}(2)\mid_{\P(V)} \simeq \Omega^1_{\P (V)}(2) \oplus \OO_{\P(V)}(1)
\]
(use the Euler sequence). Hence $\VV_{\sigma_h}$ restricts to a Gushel-Mukai vector bundle on $\P (V)$. 
\end{proof}

\begin{remark}\label{rDiscriminant}
It is known that the discriminants of the conic bundles $\CC_{\P^3}$ appearing in the proof of Proposition \ref{pGMasConicBundle} are codimension two linear sections of certain sextic fourfolds in $\P^5$, known as Eisenbud-Popescu-Walter sextics. See \cite{EPW01}, \cite{DK15-1}, \cite{DK16-1}, \cite{OGra06}. Thus the discriminants of the conic bundles $\CC_{\P^3}$ are certain nodal sextic surfaces.
\end{remark}

From the point of view of the conditions of Theorem \ref{tDegConicBundles}, we have to look for degenerations of GM conic bundles where the discriminant $\Delta$ breaks up into two cubic surfaces $\Delta_1$ and $\Delta_2$. This is so because requirements b) and c) in Theorem \ref{tDegConicBundles}, together with purity results for ramification, see e.g. \cite[Prop. 2.1, Cor. 2.2]{ABBP16}, imply that the two discriminant components $\Delta_1$ and $\Delta_2$ (whose degrees add up to six) must have nontrivial double covers that ramify only in singular points of $\Delta_1$ or $\Delta_2$.

Hence, natural candidates for $\Delta_1$ and $\Delta_2$ are cubic surfaces that themselves admit linear symmetric determinantal representations, hence occur as discriminants of (graded-free) conic bundles over $\P^3$. Since we want to exclude rank $0$ points, or entire $\P^2$ fibers, in these conic bundles, we have to disregard cones over plane cubic curves, see \cite[Rem. 9.3.10]{Dol12}, and to satisfy Theorem \ref{tDegConicBundles}, d), we have to exclude nonnormal cubic surfaces as well. This leaves us with the so-called del Pezzo cubic symmetroids  see \cite[\S 9.3.3]{Dol12} and \cite[Rem. 9.3.11]{Dol12}, see also \cite{Pio06} and \cite{BW79} as candidates for $\Delta_1$ and $\Delta_2$.  These are: the Cayley cubic $\Delta_{4A_1}$ with four $A_1$ singularities, a cubic $\Delta_{2A_1+A_3}$ with two 
$A_1$ and one $A_3$ singularity, and a cubic $\Delta_{A_5+A_1}$ with one $A_5$ and one $A_1$ singularity.

\begin{definition}\label{dAdmissibleCubics}
Any cubic $\Delta \subset \P^3$ projectively equivalent to one of \[ \Delta_{4A_1}, \; \Delta_{2A_1+A_3}, \; \Delta_{A_5+A_1}\] will be called a \emph{del Pezzo cubic symmetroid}.
\end{definition}

\begin{remark}\label{rLinesOnAdmissible}
By \cite[p. 448, Table 9.1]{Dol12}, the number of lines contained in the cubics $\Delta_{4A_1}, \; \Delta_{2A_1+A_3}, \; \Delta_{A_5+A_1}$ is $9, \: 5, \: 2$. 
\end{remark}

Note that del Pezzo cubic symmetroids come with natural double covers ramified only over the singular points.

Hence we arrive at the following notion.

\begin{definition}\label{dAdmissibleDegeneration2}
An \emph{admissible degeneration} of a Gushel-Mukai fourfold consists of
\begin{enumerate}
\item[(1)]
A GM vector bundle $\EE_{\sigma_0}$ on $\P^3$ sitting in a sequence
\[
\xymatrix{
0 \ar[r] & \OO_{\P^3} \ar[r]^{\sigma_0\quad\quad\quad} & \OO_{\P^3}(1) \oplus \Omega^1_{\P^3}(2) \ar[r] & \EE_{\sigma_0} \ar[r] & 0.
}
\]
\item[(2)]
A symmetric map
\[
\varphi_0 \colon \EE_{\sigma_0}^{\vee} \to \EE_{\sigma_0}
\]
yielding a conic bundle
\[
\xymatrix{
\CC_{\varphi_0, \sigma_0} \ar@{^{(}->}[r]\ar[rd] & \P (\EE_{\sigma_0}^{\vee})\ar[d]\\
  & \P^3
}
\]
satisfying the properties:
\begin{enumerate}
\item
The discriminant $\Delta$ of the above conic bundle splits as $\Delta=\Delta_1 \cup \Delta_2$ with $\Delta_1, \Delta_2$ del Pezzo cubic symmetroids. 
\item
Each of $\Delta_1$ and $\Delta_2$ is smooth along $\Gamma :=\Delta_1\cap \Delta_2$.
\item
The rank of the conics drops to $1$ only in finitely many points of $\Delta$.
\end{enumerate}
\end{enumerate}
\end{definition}
Condition (2), (c) is natural because we do not want the double covers of $\Delta_1$ and $\Delta_2$ to ramify in any curve, and all the examples treated in \cite{ABBP16} have this property. 

In the following, to ease notation, we will frequently drop subscripts in this set-up and simply write 
\[
\EE = \EE_{\sigma_0}, \quad \CC = \CC_{\varphi_0, \sigma_0}.
\]

Now, in the set-up of an admissible degeneration, call $\bar{\sigma}$ the composition
\[
\xymatrix{
\OO_{\P^3} \ar[r]^{\sigma_0\quad\quad\quad} & \OO_{\P^3}(1) \oplus \Omega^1_{\P^3}(2) \ar[r]^{\quad \mathrm{pr}} & \Omega^1_{\P^3}(2).
}
\]
Then $\bar{\sigma}$ does not vanish because if it did, it would vanish on a line in $\P^3$, hence $\sigma_0$ would vanish in at least some points. Hence we have a null-correlation bundle $\NN =\NN_{\bar{\sigma}}$ together with a symmetric map
\[
\psi : \NN^{\vee} \to \NN
\]
sitting in a diagram
\begin{gather}\label{dContactSurface}
\xymatrix{
\EE^{\vee } \ar[r]^{\varphi_0} & \EE\ar@{->>}[d] \\
\NN^{\vee} \ar@{^{(}->}[u]\ar[r]^{\psi}  &  \NN
}
\end{gather}
Note that $\EE=\EE_{\sigma_0}$ could, for example, split as $\OO_{\P^3}(1) \oplus \NN$ for some null-correlation bundle $\NN$; this happens if the image of $\sigma_0$ is contained in $\Omega^1_{\P^3}(2)$.

We will denote the degeneracy locus of $\psi$ by $K$. We need a couple of facts about the arrangement in $\P^3$ given by $K$, $\Delta_1$, $\Delta_2$. 

\begin{lemma}\label{lTypeResNull}
Consider a null-correlation bundle $\NN$ on $\P^3$ as above and a line $L \subset \P^3$. Then
\[
	\NN|_L = \OO(1)\oplus \OO(1) \quad\quad \text{or} \quad\quad \NN|_L = \OO(2) \oplus \OO.
\]
\end{lemma}

\begin{proof}
First we consider the Euler sequence on $\P^3$: 
\[
	0 \to \Omega(2) \to 4 \OO(1) \xrightarrow{(x_0,x_1,x_2,x_3)} \OO(2) \to 0.
\]
This sequence remains exact when restricted to $L$. After a change of coordinates we may assume
that $L = \{x_2=x_3=0\}$. We obtain
\[
	0 \to \Omega(2)|_L \to 4 \OO_L(1) \xrightarrow{(x_0,x_1,0,0)} \OO_L(2) \to 0.
\]
Now $\Omega(2)|_L$ is the kernel of the map represented by $(x_0,x_1,0,0)$, the kernel of which is
easily calculated:
\[
	0 \to \OO_L \oplus 2\OO_L(1) \xrightarrow{
	\left(
	\begin{smallmatrix} 
		x_1 & 0 & 0 \\
		-x_0& 0 & 0 \\
		0 & 1 & 0 \\
		0 & 0 & 1
        \end{smallmatrix}
        \right)
        }
	4 \OO_L(1) \xrightarrow{(x_0,x_1,0,0)} \OO_L(2) \to 0.
\]
Now $\NN$ is the cokernel of a nowhere vanishing section of $\Omega(2)$:
\[
	0 \to \OO \to \Omega(2) \to \NN \to 0.
\]
Restricting this sequence to $L$ we obtain
\[
	0 \to \OO_L \xrightarrow{\tau} \OO_L \oplus 2\OO(1) \to \NN|_L \to 0,
\]
with $\tau$ a non vanishing section. After coordinate changes there are only two possibilities
for $\tau$
\begin{enumerate}
\item $\tau = (1,0,0)$. In this case $\NN|_L = 2\OO_L(1)$
\item $\tau = (0,x_0,x_1)$. In this case $\NN|_L = \OO_L \oplus \OO_L(2)$
\end{enumerate}
\end{proof}

\begin{proposition}\label{pDegIsQuartic}
Let $\psi\colon \NN^{\vee} \to \NN$ be the symmetric map as above, derived from an admissible degeneration of Gushel-Mukai fourfolds as in Definition \ref{dAdmissibleDegeneration2}. Then the determinant of this map is a hypersurface of degree $4$. In other words, $K$ is a hypersurface of degree $4$.
\end{proposition}

\begin{proof}
We have to show that the map $\psi$ does not have rank $1$ or less on all of $\P^3$. Once we have accomplished this, we know that $K$ is a (possibly reducible or non-reduced) surface and Lemma \ref{lTypeResNull} implies that a general line $L$ intersects $K$ in four points (counted with multiplicity).

Now suppose $\psi$ dropped rank on all of $\P^3$. If the rank of $\psi$ is $1$ at a point $P$ in $\P^3$, and equal to $1$ in an entire neighborhood of $P$, then we can reduce $\varphi_0$ (and the submatrix defining $\psi$) to the following normal form locally analytically around $P$:
\[
\begin{pmatrix}
\gamma_1 & 0 & \gamma_2 \\
0 &  1 & 0\\
\gamma_2 & 0 & 0
\end{pmatrix}.
\]
Here $\gamma_i$ are holomorphic functions in a neighborhood of $P$. If $\gamma_2$ vanished at $P$, then the the discriminant of $\varphi_0$ would not be a reduced sextic, contrary to our assumption for an admissible degeneration. Hence, $\gamma_2$, hence $\det \varphi_0$, does not vanish at any point $P$ where $\psi$ has rank $1$, in other words, the sextic discriminant of $\varphi_0$ must entirely be contained in the locus of points where $\psi$ has rank $0$. However, this is absurd, because then Lemma \ref{lTypeResNull} would imply that a general line $L$ would intersect the (reduced) sextic in at most $4$ points, a contradiction. 
\end{proof}

Following \cite[Def. 1.7]{Cat81} we introduce the following piece of terminology.

\begin{definition}\label{dContact}
We say that two closed subvarieties $X_1, X_2 \subset \P^N$ have contact of order $m\ge 1$ if for every component $W$ of $X_1\cap X_2$ the intersection multiplicity of $X_1, X_2$ at $W$ is $\ge m+1$ and there is a component of $X_1 \cap X_2$ along which this intersection multiplicity is exactly equal to $m+1$. We say for short that $X_1, X_2$ have \emph{even contact} if $X_1$ and $X_2$ are contact for some $m$ and the intersection multiplicity of $X_1, X_2$ along any component of $X_1\cap X_2$ is even. 
\end{definition}

\begin{lemma}\label{lContactSeparately}
Suppose given an admissible degeneration. Then:
\begin{enumerate}
\item
Suppose that $K$ is a normal (reduced and irreducible) surface. Then 
$K$ and $\Delta$ have even contact. Indeed, $K$ and $\Delta_1$ have even contact, and $K$ and $\Delta_2$ have even contact, too. In other words, $\Delta$ cuts out a Weil (even Cartier) divisor on $K$, which we can write as $2C$ for some Weil divisor $C$ on $K$; and $\Delta$ similarly cuts out divisors of the form $2C_1$ on $\Delta_1$ and $2C_2$ on $\Delta_2$. We endow the Weil divisors $C, C_1, C_2$ with scheme structures as follows: outside of the nodes of $K$, these divisors are Cartier, hence carry natural scheme structures via the local equations defining the divisors. We define scheme-structures on $C, C_1, C_2$ simply by taking the respective closures of these schemes defined on the complement of the nodes.
\item
As a scheme, $C$ is equal to the degeneracy locus of the natural map
\[
\NN^{\vee} \to \EE ,
\]
obtained by composing the inclusion $\NN^{\vee} \to \EE^{\vee}$ with $\varphi_0$. 
\end{enumerate}
\end{lemma}

\begin{proof}
Let us prove a): clearly, $K \cap \Delta$ is some effective Weil divisor on $K$, and we want to show that it can be written as $2C$ for some other effective Weil divisor $C$. Suppose that $C'$ is an irreducible component of the intersection $K \cap \Delta$. By condition (2), (c) of Definition \ref{dAdmissibleDegeneration2}, the rank of the conics is $2$ in a general point $P$ of $C'$. Now consider diagram (\ref{dContactSurface}) at $P$: 
\begin{gather}\label{dContactSurface2}
\xymatrix{
(\EE^{\mathrm{an}})^{\vee }_{P} \ar[r]^{\varphi_{0,P}} & \EE^{\mathrm{an}}_{P} \ar@{->>}[d] \\
(\NN^{\mathrm{an}}_P)^{\vee} \ar@{^{(}->}[u]\ar[r]^{\psi_P}  &  \NN^{\mathrm{an}}_P 
}
\end{gather}
where $\EE^{\mathrm{an}}_{P}$ denotes the stalk of the analytification of $\EE$ at $P$, so is just isomorphic to $\HH_P^3$, with $\HH_P$ the stalk of the sheaf of holomorphic functions at $P$, and the other entries in the diagram are then self-explanatory. 
We can choose a basis in $\EE^{\mathrm{an}}_{P}$ and corresponding dual basis in $(\EE^{\mathrm{an}})^{\vee }_{P}$ such that $\varphi_{0,P}$ is given by a $3 \times 3$ matrix
\[
	M = \begin{pmatrix} 1 & 0 & 0 \\ 0 & 1 & 0 \\ 0 & 0 & \delta \end{pmatrix}
\]
where $\delta$ is a local equation for $\Delta$ at $P$. If $P$ is a point where $\Delta_1$ and $\Delta_2$ meet, we can write $\delta = st$ with $s$ and $t$ local equations for $\Delta_1$ and $\Delta_2$ at $P$. 

The surface $K$ is defined by the determinant of a $2 \times 2$ matrix
\[
	NMN^t
\]
with $N$ a $2 \times 3$ matrix, representing the map $ \EE^{\mathrm{an}}_{P}\to  \NN^{\mathrm{an}}_P$, which has full rank $2$ at $P$. After appropriate row transformations, respectively, choosing an appropriate basis in $\NN^{\mathrm{an}}_P$, we
can assume
\begin{align*}
	N &= \begin{pmatrix} 1 & 0 & f \\ 0 & 1 & g \end{pmatrix}\\
\mathrm{or}\; N&= \begin{pmatrix} 1 & f & 0 \\ 0 & g & 1 \end{pmatrix}\\
\mathrm{or}\; N &= \begin{pmatrix} f & 1 & 0 \\ g & 0 & 1 \end{pmatrix} 	
\end{align*}
with $f, g$ local holomorphic functions at $P$.

In the first case we have
\begin{align*}
	NMN^t &= \begin{pmatrix} 1 & 0 & f \\ 0 & 1 & g \end{pmatrix}  
	\begin{pmatrix} 1 & 0 & 0 \\ 0 & 1 & 0 \\ 0 & 0 & \delta \end{pmatrix}
	\begin{pmatrix} 1 & 0 \\ 0 & 1 \\ f & g \end{pmatrix} \\
	&= \begin{pmatrix} 
		1 & 0 & \delta f \\
		0 & 1 & \delta g
	      \end{pmatrix}
	\begin{pmatrix} 1 & 0 \\ 0 & 1 \\ f & g \end{pmatrix} \\	   
	&= \begin{pmatrix}
	    1+\delta f^2 & \delta fg \\
	    \delta fg & 1 + \delta g^2
	   \end{pmatrix}    	
\end{align*}
In particular the determinant is $1$ on $\delta =0$. This means that in this situation $K$ does not pass through $P$. 

The second and third cases lead to similar calculations and the same conclusion, so we only set down the details in the third case. We have
\begin{align*}
	NMN^t &= \begin{pmatrix} f & 1 & 0  \\ g & 0 & 1 \end{pmatrix}  
	\begin{pmatrix} 1 & 0 & 0 \\ 0 & 1 & 0 \\ 0 & 0 & \delta \end{pmatrix}
	\begin{pmatrix} f & g  \\ 1 & 0 \\ 0 & 1  \end{pmatrix} \\
	&= 
	\begin{pmatrix}
	f & 1 & 0 \\
	g & 0 & \delta
	\end{pmatrix}
	\begin{pmatrix} f & g  \\ 1 & 0 \\ 0 & 1  \end{pmatrix} \\
	&=
	\begin{pmatrix}
	f^2+1 & fg \\ 
	fg & g^2 +\delta
	\end{pmatrix}
\end{align*}
In particular, the surface $K$ is defined by 
\[
	\det(NMN^t)= f^2g^2+f^2\delta +g^2+\delta -f^2g^2 = (f^2+1)\delta + g^2
\]
and it either does not pass through $P$ or $\Delta$ cuts out twice some effective Weil divisor on $K$ locally near $P$. Moreover, if $\delta = st$, we see that $K$ has even contact with each of $\Delta_1$ and $\Delta_2$ separately. This proves the assertions in a).

\medskip

Finally, let us prove b): for this it suffices to remark that, using the above local calculation, in the third case above, the map $\NN^{\vee} \to \EE$ is locally around $P$ represented by the matrix
\[
\begin{pmatrix} f & g  \\ 1 & 0 \\ 0 & \delta  \end{pmatrix}
\]
whence, locally around $P$, the degeneracy scheme is given by $\delta= g=0$ whereas $2C$ is defined by $\delta= g^2=0$. Hence, the degeneracy scheme, being Cohen-Macaulay and thus pure-dimensional, defines $C$ scheme-theoretically. 
\end{proof}

\begin{definition}\label{dKummerETC}
A \emph{Kummer surface} in $\P^3$ is an irreducible quartic surface in $\P^3$ with $16$ nodes and no other singularities. A \emph{null-correlation quartic} in $\P^3$ is any degeneracy scheme of a symmetric map
\[
\psi \colon \NN_{\bar{\sigma}}^{\vee} \to \NN_{\bar{\sigma}}
\]
of a null-correlation bundle $\NN_{\bar{\sigma}}$. 
\end{definition}

\begin{remark}\label{rKummer}
Standard references on Kummer surfaces in $\P^3$ are \cite[Chapter10]{Dol12}, \cite{G-D94/1}, \cite{G-D94/2}, \cite{Hu05}. Blowing up the nodes, one obtains a smooth Kummer K3 surface. Null-correlation quartics are special degenerations of Kummer surfaces. One can show that every Kummer surface can be written as a null-correlation quartic in six different ways, a result we plan to elaborate on in a broader context in a future article. 
\end{remark}

\begin{remark}\label{rNoLines}
In \cite[\S 4]{G-D94/2} it is proven that a Kummer quartic in $\P^3$ contains no lines. 
\end{remark}

\begin{definition}\label{dNullCorrelationQuartic}
In the set-up of an admissible degeneration, we will call the degeneracy locus $K$ of $\psi$ the \emph{associated null-correlation quartic}.
\end{definition}

\begin{definition}\label{dTameDegeneration2}
If in the set-up of an admissible degeneration, $K$ is a Kummer surface, then we call the degeneration a \emph{tame degeneration}.  
\end{definition}

We first assume only that we are dealing with admissible degenerations of GM fourfolds. So we do not assume that $K$ is Kummer for the time being.

\begin{lemma}\label{lContactK}
Fix an admissible degeneration. Let $C$ be the degeneracy locus of the map
\[
\NN^{\vee} \to \EE ,
\]
i.e. the contact curve. 
Then $C$ is a curve of degree $12$ and arithmetic genus $p_a(C)=15$.

%\medskip

%{\color{red} maybe we do not use this extra assertion anyway and could omit it}
%Conversely, given an arrangement of two admissible cubics $X_1$ and $X_2$ in $\P^3$ both of which are contact to a reduced null-correlation quartic surface $K$ along a curve $C=C_1\cup C_2$ that can be expressed as a degeneracy locus of a map of the above type, then there exists a tame degeneration of GM fourfolds having $X=X_1\cup X_2$ as discriminant.
\end{lemma}

\begin{proof}
Consider the exact sequence  
\begin{gather}\label{fIdealSeq}
\xymatrix{
0 \ar[r] & \NN^{\vee} \ar[r] & \EE \ar[r] & \II_C (a) \ar[r] & 0
}
\end{gather}
where $\II_C$ is the ideal sheaf of $C$, and $a$ is some twist. We will compute $a$ and the arithmetic genus $\chi (\II_C)$ via a Chern class computation and Hirzebruch-Riemann-Roch. See \cite[Appendix A]{Ha77} for the relevant background. Here we just recall the following for a vector bundle $\VV$ of rank $r$ on $\P^3$ and a codimension $z$ subscheme $Z$ of $\P^3$ and vector bundle $\WW_Z$ of rank $n$ on $Z$:
\begin{align}
\label{fDual}
c_r (\VV ) &= (-1)^r c_r (\VV ) \\ \label{fZ}
c_i (\WW_Z) &=0 \quad i< z, \quad c_z (\WW_Z) = (-1)^{z-1} (z-1)! n[Z]\\
\label{fChernChar}
\mathrm{ch}(\VV ) &= r + c_1 (\VV ) + \frac{1}{2} \left( c_1(\VV )^2 - 2c_2 (\VV )\right) + \frac{1}{6} \left( c_1(\VV )^3 - 3 c_1(\VV )c_2 (\VV ) + 3c_3 (\VV )\right)\\
\label{fToddClass}
\mathrm{td} (\VV )&= 1 + \frac{1}{2} c_1 (\VV ) + \frac{1}{12}\left( c_1 (\VV )^2 + c_2 (\VV )\right) + \frac{1}{24} c_1 (\VV ) c_2 (\VV )\\
\label{fHRR}
\chi (\VV ) &= \deg \left( \mathrm{ch} (\VV ) . \mathrm{td}(\TT_{\P^3})\right)_3
\end{align}
where $\mathrm{ch}$ is the Chern character, $\mathrm{td}$ is the Todd class, $\chi$ the Euler characteristic, and all computations are in the Chow ring of $\P^3$ with rational coefficients. Note that as usual, from the Chern character one obtains a homomorphism 
\[
\mathrm{ch} \colon K_0 (\P^3) \to A^* (\P^3)_{\Q}
\]
which is a homomorphism of rings. Note that $A^* (\P^3)_{\Q}$ is a truncated polynomial ring in the class $h$ of a hyperplane. 
In particular, $\mathrm{ch}$ is additive on short exact sequences. Now the Euler exact sequence
\begin{gather}\label{fEulerSeq}
\xymatrix{
0 \ar[r] & \OO \ar[r] & 4\OO (1) \ar[r] & \TT \ar[r] & 0
}
\end{gather}
gives 
\[
\mathrm{ch}(\TT ) = 3 + 4h + 2 h^2 + \frac{2}{3} h^3, \quad c_1 (\TT ) = 4h, \: c_2 (\TT )= 6h^2, \: c_3 (\TT ) = 4h^3. 
\]
In particular, 
\begin{gather}\label{fToddClassP3}
\mathrm{td}(\TT_{\P^3}) = 1 + 2h + \frac{11}{6} h^2 + h^3.
\end{gather}
Now using the exact sequences
\begin{align}
\label{fEulerDual}
\xymatrix{
0 \ar[r] & \Omega^1(2) \ar[r] & 4\OO (1) \ar[r] & \OO (2) \ar[r] & 0
}\\
\label{fExactSeqE}
\xymatrix{
0 \ar[r] & \OO \ar[r] & \OO (1)\oplus \Omega^1 (2) \ar[r] & \EE \ar[r] & 0
}\\
\label{fEulerTwist}
\xymatrix{
0 \ar[r] & \OO (-2) \ar[r] & 4\OO (-1) \ar[r] & \TT (-2) \ar[r] & 0
}\\
\label{fSeqNDual}
\xymatrix{
0 \ar[r] & \NN^{\vee} \ar[r] & \TT (-2) \ar[r] & \OO \ar[r] & 0
}
\end{align}
and the additivity of the Chern character we get
\begin{align}\label{fChernE}
\mathrm{ch}(\EE ) &= 3 + 3h + \frac{1}{2} h^2 - \frac{1}{2} h^3\\
\label{fChernNDual}
\mathrm{ch}(\NN^{\vee}) &= 2 - 2h + \frac{2}{3} h^3. 
\end{align}
Now using the initial sequence (\ref{fIdealSeq}) we get that
\begin{gather}\label{fChernIdeal}
\mathrm{ch}(\II_C (a)) = 1 + 5h + \frac{1}{2} h^2 - \frac{7}{6} h^3 .
\end{gather}
Since there is an exact sequence 
\[
\xymatrix{
0 \ar[r] & \II_C (a)  \ar[r] & \OO (a) \ar[r] & \OO_C (a) \ar[r] & 0
}
\]
and by formula (\ref{fZ}) and the additivity of the Chern character, the first Chern class of $\II_C (a)$ is equal to the first Chern class of $\OO (a)$, which is $ah$. Thus $a=5$. Again, by formula (\ref{fZ}), $c_2 (\OO_C (a))= - \deg(C) h^2$, and this, by the additivity of the Chern character and formula (\ref{fChernChar}), is equal to $- (25/2 - 1/2)h^2= -12h^2$, which checks. Moreover, since the Chern character is a ring homomorphism we get
\[
\mathrm{ch} ( \II_C ) = \mathrm{ch} (\II_C(5) \otimes \OO (-5)) = \mathrm{ch} (\II_C(5) \otimes^L \OO (-5)) = \mathrm{ch} (\II_C(5)). \mathrm{ch}(\OO (-5)) .
\]
Hence, by the Hirzebruch-Riemann-Roch formula (\ref{fHRR}) and formula (\ref{fToddClassP3}), we have that the arithmetic genus is the coefficient in front of $h^3$ in 
\begin{gather*}
\mathrm{ch} (\II_C(5)). \mathrm{ch}(\OO (-5)). \mathrm{td}(\TT_{\P^3}) = \\
      = \left( 1 + 5h + \frac{1}{2} h^2 - \frac{7}{6} h^3 \right) \cdot \left( 1- 5h + \frac{25}{2} h^2 - \frac{125}{6} h^3\right)\cdot \left( 1 + 2h + \frac{11}{6} h^2 + h^3\right) \\
      = 1 + 2h - \frac{61}{6} h^2 + 15 h^3. 
\end{gather*}
Hence $p_a (C) = 15$.
\end{proof}

Let us \textbf{assume from now on that $K$ is a Kummer surface}. So we work now with a tame degeneration of GM fourfolds. 

Let $\pi\colon \tilde{K} \to K$ be the minimal resolution of the nodes, let $E_1, \dots , E_{16}$ be the resulting $(-2)$ curves. Then $\pi^*\omega_K = \omega_{\tilde{K}}$ is trivial, and the arithmetic genus of a divisor $D\subset \tilde{K}$ is
\[
\frac{D^2}{2} +1.
\]
Recall that $\Delta = \Delta_1 \cup \Delta_2$ is the union of two del Pezzo cubic symmetroids such that $\Delta_1\cap \Delta_2$ consists of smooth points on both $\Delta_1$ and $\Delta_2$. 

%\begin{definition}
%We say that $\Delta$ is contact to $K$ in a \emph{simple contact curve} $C$ if $\Delta$ is smooth in all nodes of $K$. 
%\end{definition}

%\begin{remark}
%For a simple contact curve $C$, denoting the strict transform of $C$ on $\tilde{K}$ by $\bar{C}$ we have
%\[
%p_a (C) = p_a (\bar{C}).
%\]
%\end{remark}

We need one more preparatory result on the behavior of the arithmetic genus of curves in our situation. 

\begin{lemma}\label{lArithmeticGenusDrop}
Let $C$ be a purely one-dimensional algebraic subscheme of $K$, that is, $C$ has no embedded points, but may be nonreduced or reducible. Let $\bar{C}$ be its strict transform on $\tilde{K}$, that is, the scheme-theoretic closure of $\pi^{-1}\mid_{\tilde{U}}(C\cap U)$ in $\tilde{K}$ where $U$ is the complement of the nodes in $K$, and $\tilde{U}$ is the complement of the $(-2)$-curves in $\tilde{K}$. 
Let $\nu$ be the number of nodes of $K$ in which the scheme $C$ is singular. Then 
\[
p_a (\bar{C}) \le p_a (C) - \nu .
\]
\end{lemma}

\begin{proof}
The projection $\pi$ induces a morphism of schemes $\pi_{\bar{C}}\colon \bar{C} \to C$; namely, clearly $\bar{C}$ is contained in the scheme-theoretic preimage of $C$ under $\pi$. 
We have an exact sequence
\[
\xymatrix{
0 \ar[r] & \OO_C \ar[r] & \pi_{*}\OO_{\bar{C}} \ar[r] & \FF \ar[r] & 0
}
\]
with $\FF$ supported in the nodes. Note that for the injectivity of the arrow $\OO_C \to \OO_{\bar{C}}$ we use that $C$ has no embedded points. We want to show that the stalk of $\FF$ is nonzero at every node $p$ of $K$ where $C$ is singular. This can be checked locally analytically. We distinguish two cases. Either the reduction $C^{\mathrm{red}}$ of $C$ is already singular at the node. Then $\pi_C$ can not be an isomorphism locally above $p$ because otherwise the associated reductions $\bar{C}^{\mathrm{red}}$ and $C^{\mathrm{red}}$ would be isomorphic as well, but this would contradict the fact that we can resolve the singularities of $C^{\mathrm{red}}$ by an embedded resolution, blowing up a sequence of (infinitely near) points in $\P^3$ successively. 
In the second case, $C$ has smooth reduction $C^{\mathrm{red}}$ at $p$, but the scheme $C$ is singular there. Hence we can assume that $C$ is a line of a ruling on a quadric cone $Q$ in $\P^3$, with a multiple structure $m > 1$. We only have to show that its strict transform on the blowup $\tilde{Q}$ of $Q$ at the vertex has strictly smaller arithmetic genus. 
Now $\tilde{Q}$ is isomorphic to the Hirzebruch surface 
\[
\F_2 = \P \left( \OO_{\P^1}\oplus \OO_{\P^1} (2)\right)
\]
and $\bar{C}$ becomes a multiple fiber, hence has self-intersection zero; the canonical class of $\F_2$ is $-2b-4f$ where $b$ is the class of the negative section, which is the $(-2)$-curve, and $f$ is a fiber. Hence the arithmetic genus of $\bar{C}$ on $\F_2$ is $1-m$. To estimate the arithmetic genus of $C$ on $Q$ we use the following observation: the arithmetic genus of $\bar{C}$ is that of $m$ disjoint $\P^1$'s. Now let $L^{m}\subset Q$ be a union of $m$ different lines of the ruling of $Q$. Let $\nu\colon \tilde{L}^m \to L^m$ be the normalization, a union of $m$ disjoint $\P^1$'s. The exact sequence 
\[
\xymatrix{
0 \ar[r] & \OO_{L^m} \ar[r] & \nu_{*}\OO_{\tilde{L}^m} \ar[r] & \GG \ar[r] & 0
}
\]
shows that $L^m$ has arithmetic genus $> 1-m$. Now we consider a flat family $L^m_t$, $0 < t < \epsilon$ of such unions of $m$ lines of the ruling of $Q$ such that for $t\to 0$ the $m$ lines come together. Let $L^m_0$ be the flat limit of this one-parameter family of schemes. It coincides with the scheme $C$, whose arithmetic genus we are interested in estimating, outside of the vertex of the cone $Q$, but it may have an embedded component supported at the vertex. Since the arithmetic genus is constant in flat families, $p_a (L^m_0) = p_a (L^m_t) > 1-m$. But $p_a (C) \ge p_a (L^m_0)$ because of the exact sequence of sheaves on $Q$
\[
\xymatrix{
0 \ar[r] & \KK \ar[r] & \OO_{L^m_0} \ar[r] & \OO_{C} \ar[r] & 0
}
\]
where $\KK$ is supported at the vertex in case there is an embedded component. In any case, $p_a (C) > 1-m$, which is what we wanted to show.
\end{proof}

\begin{lemma}\label{lCurvesType}
Assume $K$ is a Kummer surface with desingularization $\pi\colon \tilde{K} \to K$ as before, and that $\Delta$ is a sextic as above.  Then the class of the strict transform $\bar{C}$ in $\mathrm{Pic}(\tilde{K})$ of the contact curve $C$ is equal to
\[
\bar{C}\equiv \frac{1}{2} \left( 6H - E_1 - \dots - E_{16} \right)
\]
where $H$ is the pull back of the hyperplane class $H$ in $\P^3$ and $\equiv$ denotes linear equivalence. 
In particular, $\Delta$ and $K$ have even contact in such a way that $\Delta$ is smooth in every node of $K$. 
\end{lemma}

\begin{proof}

Clearly the strict transform $\bar{C}$ of $C$ is a divisor on $\tilde{K}$ that satisfies
\[
\frac{1}{2} \left( 6H - a_1E_1 - \dots - a_{16}E_{16} \right)
\]
for some positive integers $a_1, \dots , a_{16}$. Moreover, $a_i$ is equal to the multiplicity of $\Delta$ in the node of $K$ corresponding to $E_i$, hence by our assumption, each $a_i$ is at most $2$ since the $\Delta_i$ individually only have rational double point singularities, and $\Delta_1\cap \Delta_2$ consists of smooth points of both $\Delta_1$ and $\Delta_2$. Note that here we are using property (b) of Definition \ref{dAdmissibleDegeneration2}, that each of $\Delta_1$ and $\Delta_2$ are smooth in each point of $\Delta_1\cap \Delta_2$, in an essential way: otherwise, one of the $a_i$ might be $3$, for example, and the following argument fails. 

The arithmetic genus of $C$ is $15$ and the arithmetic genus of $\bar{C}$ can only decrease by Lemma \ref{lArithmeticGenusDrop}, so by the genus formula
\[
15\ge p_a (\bar{C}) = \frac{1}{2} \left( 6^2 - \frac{1}{2} \sum_{i=1}^{16} a_i^2 \right) +1.
\]
Moreover, for every $i$ with $a_i=2$, the arithmetic genus of $\bar{C}$ decreases by at least $1$ compared to that of $C$ by Lemma \ref{lArithmeticGenusDrop} because then $C$ is singular at the corresponding node.  
Hence denoting $I^{0,1}\subset \{ 1, \dots , 16 \}$ the subset of indices for which $a_i$ is zero or one, we obtain
\begin{gather*}
\sum_{i\in I^{0,1}} a_i^2 \ge 16
\end{gather*}
which is absurd unless all of the original $a_i$ are equal to $1$. 
Hence we get 
\begin{gather}\label{fSumCoeff}
\sum_{i=1}^{16} a_i^2 =16
\end{gather}
and both assertions.
\end{proof}

%In view of Lemma \ref{lContactSeparately} and the fact that by Lemma \ref{lCurvesType}, $\Delta$ is smooth in all the nodes of $K$, we get that
%\[
%\Delta_1 \cap K = 2C_1, \quad \Delta_2\cap K = 2 C_2
%\]
%for some pure one-dimensional schemes $C_1$, $C_2$ in $\P^3$. We can view $C_1$ and $C_2$ as Weil divisors on $K$, and we will denote their strict transforms on $\tilde{K}$ by $\bar{C}_1$ and $\bar{C}_2$, which are Cartier divisors. 
%Note that $C_1\cup C_2= C$.

\begin{lemma}\label{lCurvesBreakUp}
We have, after possibly renumbering the $E_i$ and interchanging the roles of $\Delta_1$ and $\Delta_2$, the following equality in $\mathrm{Pic}(\tilde{K})$: 
\begin{align*}
\bar{C}_1 & \equiv \frac{1}{2} \left( 3H - E_1 - \dots - E_{10} \right), \\
\bar{C}_2 & \equiv \frac{1}{2} \left( 3H - E_{11} - \dots - E_{16} \right) . 
\end{align*}
Moreover, the purely one-dimensional subschemes $C_1$ and $C_2$ of $K$ have degree $6$ and $p_a (\bar{C}_1) = 3$ and $p_a (\bar{C}_2) = 4$. 
\end{lemma}

\begin{proof}
Suppose $\Delta$ decomposes into two del Pezzo cubic symmetroids as in an admissible degeneration. Then we have
\begin{align*}
\bar{C}_1 & \equiv \frac{1}{2} \left( 3H - \sum_{i\in I} b_i E_i \right), \\
\bar{C}_2 & \equiv \frac{1}{2} \left( 3H - \sum_{j\in J} c_j E_j  \right) . 
\end{align*}
with $I\sqcup J =\{ 1, \dots , 16\}$. Moreover, each $b_i$ and $c_j$ is either $0$ or $1$. The genus of $\bar{C}_1$ is then equal to 
\[
p_a (\bar{C}_1) = \frac{1}{2} \left( 9 - \frac{1}{2} \sum_{i\in I } b_i^2 \right) +1 = \frac{11}{2} - \frac{1}{4} \sum_{i\in I } b_i^2. 
\]
For this to be an integer, we only have the following possibilities for the cardinality of $I$:
\[
| I | = 2, \: 6, \: 10, \: 14.
\]
The corresponding possible genera are
\[
p_a (\bar{C}_1) = 5,\: 4, \: 3, \: 2. 
\]
Note that the same restrictions hold for $|J|$ and $p_a (\bar{C}_2)$. Together with the requirement that $|I\sqcup J|=16$, it implies that we can only have, after possibly interchanging the roles of $C_1$ and $C_2$, the following two possibilities:
\begin{gather}\label{fPossibleCurves}
p_a (\bar{C}_1) = 3, |I|=10, \quad p_a (\bar{C}_2) = 4, |I|=6, \\ \nonumber
p_a (\bar{C}_1) = 5, |I|=2, \quad p_a (\bar{C}_2) = 2, |I|=14 .
\end{gather}
The first line in formula (\ref{fPossibleCurves}) leads to the situation described in the Lemma, thus we only have to rule out that $p_a (\bar{C}_1) =5$ can occur. In that case $C_1$ would be a sextic that has arithmetic genus bigger than or equal to $5$ (by Lemma \ref{lArithmeticGenusDrop}, since $C_1$ is pure-dimensional by its definition) lying on an irreducible cubic surface $\Delta_1$. 
This is impossible by Theorem \ref{tSexticCurves} in Appendix \ref{aBothmer}. 
\end{proof}

%\begin{definition}\label{dFormat}
%We call a variety $Z\subset \P^3$ determinantal of type
%\[
%\begin{pmatrix}
%d_{11} & \dots & d_{1n}\\
%\vdots  & \ddots & \vdots \\
%d_{m1} & \dots & d_{mn}
%\end{pmatrix}
%\]
%for nonnegative integers $d_{ij}$ if there exists a $m\times n$ matrix of homogeneous polynomials $M$ of degrees as specified in the matrix above such that $Z$ is the scheme defined by the maximal minors of $M$ and $Z$ is of expected codimension $|m-n|+1$.
%\end{definition}

\begin{definition}\label{dTropes}
A \emph{trope} on a Kummer surface is a conic $T$ on $K$ such that $2T$ is cut out on $K$ by some plane $H$ in $\P^3$ passing through $6$ of the nodes on $K$. To distinguish $H$ from $T$, we will call $H$ a \emph{trope plane} in this case.
\end{definition}

This is equivalent to \cite[Def. 1.11]{G-D94/1}. 

\begin{remark}\label{rTropesIncidence}
It is classically known, see e.g. \cite[10.3]{Dol12}, \cite{Hu05}, \cite{G-D94/1} or \cite{G-D94/2}, that there are sixteen tropes on a Kummer surface $K$ forming a \emph{$(16, 6)$ configuration} with the $16$ nodes of $K$: this means that every trope plane contains exactly six of the nodes, and every node lies on exactly six trope planes.
Moreover, for every Kummer surface, the configuration of nodes and trope planes is \emph{non-degenerate} \cite[Prop. 1.12]{G-D94/1}, meaning that every two trope planes share exactly two nodes, and every pair of nodes is contained in exactly two trope planes. 

Moreover, the entire incidence relations between nodes and tropes on a Kummer surface can be very compactly displayed by the following diagrams, see \cite[Ch. 1, \S 5]{Hu05}, \cite[p. 523]{Dol12}, \cite[Lemma 1.4, Prop. 1.12]{G-D94/1}:
\[
\begin{matrix}
\times & \times & \times & \times \\
\times & \times & \times & \times \\
\times & \times & \times & \times \\
\times & \times & \times & \times 
\end{matrix}
\quad \quad \quad 
\begin{matrix}
\bullet & \bullet & \bullet & \bullet \\
\bullet & \bullet & \bullet & \bullet \\
\bullet & \bullet & \bullet & \bullet \\
\bullet & \bullet & \bullet & \bullet 
\end{matrix}
\]
Here the $\times$ represent trope planes and the $\bullet$ represent nodes. To find out which nodes a given trope plane contains, go to the position of the node corresponding to the position of the trope plane and take all nodes in the same row and column, excluding the distinguished node. Thus in the diagram
\[
\begin{matrix}
\times & \times & \times & \times \\
\times & \sharp & \times & \times \\
\times & \times & \times & \times \\
\times & \times & \times & \times 
\end{matrix}
\quad \quad \quad 
\begin{matrix}
\bullet & \circ & \bullet & \bullet \\
\circ & \bullet & \circ & \circ \\
\bullet & \circ & \bullet & \bullet \\
\bullet & \circ & \bullet & \bullet 
\end{matrix}
\]
the nodes in trope plane $\sharp$ are the ones marked $\circ$. In particular, this illustrates the fact that the intersection of two trope planes contains exactly two nodes, and there are no three trope planes passing through a line. For example, the following diagram shows the two nodes contained in the marked trope planes:
\[
\begin{matrix}
\times & \times & \sharp & \times \\
\sharp & \times & \times & \times \\
\times & \times & \times & \times \\
\times & \times & \times & \times 
\end{matrix}
\quad \quad \quad 
\begin{matrix}
\circ & \bullet & \bullet & \bullet \\
\bullet & \bullet & \circ & \bullet \\
\bullet & \bullet & \bullet & \bullet \\
\bullet & \bullet & \bullet & \bullet 
\end{matrix}
\]
and read from right to left, dually, the only two trope planes containing the two nodes $\circ$ are the ones marked $\sharp$.
\end{remark}

\begin{lemma}\label{lDescriptionCurves}
With the set-up and notation of Lemma \ref{lCurvesBreakUp}, we have that $C_2$ is a complete intersection of a quadric and a cubic in $\P^3$. 
\end{lemma}

\begin{proof}
By Lemma \ref{lCurvesBreakUp}, we know $C_2$ is a sextic of arithmetic genus at least $4$ (since $C_2$ is pure dimensional by definition of its scheme structure and receives a surjection from $\bar{C}_2$, which has arithmetic genus $4$). Moreover, $C_2$ lies on the irreducible cubic $\Delta_2$. By Theorem \ref{tSexticCurves}, $C_2$ cannot have arithmetic genus $\ge 5$, so $p_a (C_2) =4$ and Proposition \ref{pCompleteIntersection} implies that it is a complete intersection of $\Delta_2$ and a quadric. 
\end{proof}

We will now consider the cubic $S:= \Delta_2$ that is contact to the Kummer surface $K$ in the complete intersection curve $C_2$ of type $(2,3)$. 

\begin{lemma}\label{lContact}
If an irreducible normal cubic surface $S$ is contact to a  Kummer quartic surface $K$ in a complete intersection curve of type $(2,3)$, then the equation of $K$ can be written as 
\[
	\det \begin{pmatrix} a & b \\ b & c \end{pmatrix} = 0
\]	
where $c=0$ is the equation of $S$, $b$ is quadratic and $b=c=0$ defines the contact curve, $a$ is a linear form.
\end{lemma}

\begin{proof}
If $c$ is the equation of the cubic $S$, then by hypothesis, the contact curve is defined by $c=b=0$ with $b$ quadratic. Hence, $b^2$ and the equation $k$ of $K$ cut out the same divisor on $S$, thus, since $S$ is normal and Cartier divisors inject into Weil divisors, the function $b^2/k$, defining a trivial Weil divisor, is a nonvanishing constant on $S$. Absorbing this constant into $k$, we get that $b^2-k$ is divisible by $c$, whence the assertion.
\end{proof}

Note that in the situation of Lemma \ref{lContact}, $a=0$ defines a trope on $K$. 

\begin{proposition}\label{pLinesInTropes}
Let
\[
	\det \begin{pmatrix} a & b \\ b & c \end{pmatrix} = 0
\]	
be the equation of a Kummer surface $K$ where $c=0$ the equation of an irreducible normal contact cubic $S$
to $K$ and $a=0$ is the equation of a contact plane to $S$.

Consider a second plane $\P^2 \subset \P^3$ that is also contact to $K$. If this $\P^2$
is different from $a=0$ then $\P^2 \cap S$ contains a line.
\end{proposition}

\begin{proof}
$\P^2 \cap K$ is a double conic. Therefore on this $\P^2$ we can write
\[
	\begin{pmatrix} a & b \\ b & c \end{pmatrix} = -d^2.
\]
where we denote by $a,b,c$ also their restrictions to $\P^2$ and by $d$ an appropriate equation of the
tangent conic.

We then have
\[
	ac -b^2 = -d^2 \iff  ac =  b^2-d^2 = (b +d) (b -d).
\]
Now $a$ and $c$ are nonzero on $\P^2$: $a$ by assumption and $c$ since it defines an irreducible contact cubic.
Since the equation above is in a factorial ring the degrees of the irreducible factors of $ac$ in the coordinate ring of $\P^2$
must be a subpartition of $\{3,1\}$ (for the left side) and $\{2,2\}$ for the right side, i.e either $\{2,1,1\}$ or $\{1,1,1,1\}$. In any case $c$ can not
be irreducible on $\P^2$ and contains at least one linear factor. 
\end{proof}

\begin{theorem}\label{tExclusion}
There is no tame degeneration of Gushel-Mukai fourfolds, which means in particular a degeneration where  $\Delta_1, \Delta_2$ 
are del Pezzo cubic symmetroids and $K$ is a Kummer surface in $\P^3$. 
\end{theorem}

\begin{proof}
First of all, by Lemmata \ref{lContactK}, \ref{lCurvesType}, \ref{lDescriptionCurves}, Lemma \ref{lContact} and hence Proposition \ref{pLinesInTropes} is applicable. Here $S=\Delta_2$, the cubic contact to $K$ in the genus $4$ sextic curve $C_2$. Suppose this was a del Pezzo cubic symmetroid. It has at $9,5$ or $2$ lines on it by Remark \ref{rLinesOnAdmissible}. Denote the trope defined by $a=0$ in the notation of Proposition \ref{pLinesInTropes} by $T$. Then $K$ admits $15$ tropes different from $T$, and each contains one of the lines of $S$. By the pigeonhole principle, in the cases of $5$ or $2$ lines, we would get a configuration of three tropes passing through a line, contradicting Remark \ref{rTropesIncidence}. Hence the number of lines must be $9$ and $S$ must be a Cayley cubic if it is to work at all. 

But then there are at least six lines on $S$ that are contained in $2$ of the $15$ tropes other than $T$. Each such line must pass through $2$ nodes of the Kummer surface by Remark \ref{rTropesIncidence}, of which at most one can lie on $T$, since there are no three tropes through a line, again by Remark \ref{rTropesIncidence}. Since $S$ also passes through all $6$ nodes on $T$, we see that $S$ passes through at least $7$ nodes and the class of the strict transform $\bar{C}_2$ of the contact curve can not be
\[
	\frac{1}{2}(3H - E_{11} - \dots E_{16}).
\]
So the case of a Cayley cubic is impossible, too.
\end{proof}

\begin{remark}\label{rMoreGeneral}
In computer experiments using Macaulay2, we did not find any examples of admissible degenerations of Gushel-Mukai fourfolds either (when one allows $K$ to be a null-correlation quartic with worse singularities than just $16$ nodes). Of course that is not a proof that they cannot exist. One might also allow the ambient bundles to degenerate to some sheaves.
\end{remark}

\appendix

\section{Sextic curves of arithmetic genus $\ge 5$ in $\P^3$}\label{aBothmer}

%\begin{theorem}\label{tBothmer}
%Suppose that $C$ is a one-dimensional subscheme of $\P^3$ without embedded components, but not necessarily reduced or irreducible. We will refer to this as a curve below. Then if the arithmetic genus $p_a (C) \ge 5$, then $C$ is not connected. 
%\end{theorem}

%\begin{proof}
%...
%\end{proof}

The aim of this section is to prove the following 

\begin{theorem}\label{tSexticCurves}
Let $C \subset \P^3$ be a scheme of dimension $1$, degree $6$ and arithmetic genus at least $5$. Then
$C$ can not lie on an irreducible cubic surface.
\end{theorem}

To prove this theorem we use the theory of generic initial ideals introduced by Mark Green \cite{Green98}.
It reduces the question above to the combinatorics of certain diagrams. 

\begin{notation} Let $f \colon \N^2 \to \N \cup \{\infty\}$ be a map. We depict $f$ in a diagram as follows:

\begin{center}
\begin{tabular}{ccccccccc}
\\
	&&& f(0,0) \\
	&& f(1,0) && f(0,1) \\
	& f(2,0) && f(1,1) && f(0,2) \\
	\dots && \dots && \dots && \dots  \\
	\\
\end{tabular}
\end{center}

Furthermore, $f(i,j)$ is usually replaced by $\bullet$ if $f(i,j) = 0$ and $f(i,j)$ is replaced by $\circ$ if $f(i,j) = \infty$, and we adhere to this practice.

In this situation we also set
\[
  \lambda_i := \min \{ j \,|\, f(i,j) \not= \circ \}.
\]
%Notice that the $\lambda_i$ are just the number of $\circ$'s in the down-right diagonals. See \cite[the discussion after the Example after Lemma 4.2.]{Green98}.
\end{notation}

%\begin{remark}
%Notice that the $\lambda_i$ are just the number of $\circ$'s in the down-right diagonals. See \cite[the discussion after the Example after Lemma 4.2.]{Green98}.
%For $C$ and $\Gamma$ as in Example \ref{eDeg5g2} we have $\lambda_0 = 4$ and $\lambda_1 = 1$. All other $\lambda_i$ are zero.
%\end{remark}

\begin{remark} \label{rGin}
If $C \subset \P^3$ is a $1$-dimensional scheme, then Mark Green
associates to $C$ a function $f_C$ and a diagram $\Delta(C)$ as above, in the following
way: Let 
\[
\gin(I_C) \subset \C[x_0,x_1,x_2,x_3]
\]
 be the generic initial ideal with respect to the reverse lexicographic order 
as in \cite[Remark after Theorem 1.27]{Green98}, then following \cite[Definition 4.18]{Green98} 
\[
	f_C(i,j) := \min \{k \,|\, x_0^ix_1^jx_2^k \in \gin(I_C) \}.
\]
Similarly Green 
also associates such a function $f_{\Gamma}$ and diagram $\Delta(\Gamma)$ to any $0$-dimensional scheme $\Gamma \subset \P^2$ \cite[Discussion following the Example after Lemma 4.2]{Green98}. Here only $\circ$ and $\bullet$ can occur.
% \cite[Remark after Theorem 1.27]{Green98}.
\end{remark}

\begin{remark} 
Mark Green assumes $C$ irreducible, reduced and non degenerate before Definition 4.18 of \cite{Green98}, but this
assumption is not used until Example 4.21. In this Appendix we use only facts proved before that.
\end{remark}

The diagrams of $C \subset \P^3$ and its generic hyperplane section  are related
in a nice way:

\begin{proposition} \label{pHyperplaneSection}
Let $C \subset \P^3$ be a one dimensional scheme as above and let $\Gamma \subset \P^2$ be a generic hyperplane section of $C$. Then $\Delta(\Gamma)$ is
obtained from $\Delta(C)$ by replacing all entries other than $\bullet$ and $\circ$ (i.e., all numerical entries) with $\bullet$'s.
\end{proposition}

\begin{proof}
\cite[the sentence before Corollary 4.20]{Green98}
\end{proof}

\begin{proposition} \label{pSpaceCurveDiagram}
Let $C \subset \P^3$ be a $1$-dimensional scheme. Then $f_C$ and $\Delta(C)$ have the following
properties:

\begin{enumerate}
%\item \label{iWhatGreenReallySays}
%If $f_C (i_1, i_2) < \infty$, then 
%\[
%f_C (i_1 , i_2) > f_C (i_1, i_2 +1) \ge f_C (i_1 +1 , i_2).
%\]
%{\color{red} this is what Green writes but it is incomplete and apparently not always true if $f_C (i_1, i_2)=0$?}
\item \label{iBorelFixed1} Reading every horizontal row from left to right, the entries are weakly increasing from one to the next. %{\color{red} check if this is true; in this form still unproven and I could not find a reference}
\item \label{iBorelFixed2} The entries in each north-west to south-east diagonal row, read from left to right, as well as the entires in each north-east to south-west diagonal row, read from right to left, are weakly decreasing from one to the next, and 
decrease strongly whenever the respective entry is not equal to $\bullet$ or $\circ$. %{\color{red} check if this is true; in this form still unproven and I could not find a reference}
\item \label{iDegree} The degree of $C$ is the number of $\circ$'s.
\item \label{iArithmeticGenus} The arithmetic genus $p_a$  of $C$ is obtained by adding 1 for each $\circ$ in the third row, 2 for each $\circ$ in the
fourth row, etc. (ignore the first and second rows), and then subtracting the numbers (entries different from $\bullet$ and $\circ$) in the diagram.
\item \label{iHypersurface} If $f(i,j) = k < \infty$ then $C$ lies on a (not necessarily irreducible) hypersurface of degree $i+j+k$. 
\end{enumerate}
\end{proposition}

\begin{proof}
The generic initial ideal $\gin(I_C)$ as in \cite[Remark after Theorem 1.27]{Green98} is by definition Borel-fixed 
\cite[Definition 1.26]{Green98} and therefore is closed under elementary moves \cite[Definition 1.24 and Proposition 1.25]{Green98}. In our situation this means that 
\begin{enumerate}[label=(\roman*)]
\item $x_0^ix_1^jx_2^k \in \gin(I_C), j>0 \implies x_0^{i+1}x_1^{j-1}x_2^k \in \gin(I_C)$ \label{i1to0}
\item $x_0^ix_1^jx_2^k \in \gin(I_C), k>0 \implies x_0^{i+1}x_1^{j}x_2^{k-1} \in \gin(I_C)$ \label{i2to0}
\item $x_0^ix_1^jx_2^k \in \gin(I_C), k>0 \implies x_0^{i}x_1^{j+1}x_2^{k-1} \in \gin(I_C)$ \label{i2to1}
\end{enumerate}
i.e. we can replace any variable in a monomial by one with lower index.

Property $\ref{i1to0}$ implies $\ref{iBorelFixed1})$ of the Proposition: firstly, $f(i+1, j-1)\le f(i, j)$ for all $j \ge 1$ and $f (i, j)\neq \infty$ (i.e., $\neq \circ$). Secondly, if $f(i,j)=\infty$, the property claimed in $\ref{iBorelFixed1})$ of the Proposition is vacuously true. 
Properties $\ref{i2to0}$ and $\ref{i2to1}$ imply $\ref{iBorelFixed2})$ of the Proposition if $f(i,j)\neq \infty$ (i.e., $f(i,j)\neq \circ$) and $f (i,j) \neq 0$ (i.e., $f(i,j) \neq \bullet$). In the case $f(i,j)=\infty$, it is vacuously true that $f(i+1, j)$ and $f(i, j+1)$ are both not bigger than $f(i,j)$. If $f(i,j)=0$, then also $f(i+1, j)= f(i, j+1)=0$ because $\gin (I_C)$ is an ideal. 

\medskip

For \ref{iDegree}) of the Proposition let $\Gamma$ be a generic hyperplane section of $C$. We then have
\[
	\deg(C) = \deg(\Gamma) = \sum_i \lambda_i(\Gamma) = \sum_i \lambda_i(C)
\]
where the second equality holds by \cite[the discussion after Definition 4.1]{Green98} and the third
equality follows from Proposition \ref{pHyperplaneSection}. 
Notice that the $\lambda_i$ are just the number of $\circ$'s in the down-right diagonals in any diagram satisfying \ref{iBorelFixed1}) and \ref{iBorelFixed2}). See \cite[the discussion following the Example after Lemma 4.2]{Green98}.
Hence $\sum_i \lambda_i(C)$ is just
the number of $\circ$'s in the diagram.

\ref{iArithmeticGenus}) is \cite[Remark before Example 4.21]{Green98}.

\ref{iHypersurface}) By Remark \ref{rGin}
\[
	f_C(i,j) = k < \infty \iff x_0^ix_1^jx_2^k \in \gin(I_C).
\]
Now by \cite[Theorem 1.27]{Green98} and the definition of $\gin(I_C)$ there exists a linear transformation $g$ of $\P^3$ and a homogeneous polynomial $f \in I_C$ such that $x_0^ix_1^jx_2^k$ is  the initial term of $g(f)$. Since $f$ is homogeneous we have
\[
	\deg f = \deg g(f) = \deg \initialTerm(g(f)) = \deg x_0^ix_1^jx_2^k = i+j+k.
\]
\end{proof}

\begin{example} \label{eDeg5g2}
The diagram
\begin{center}
\begin{tabular}{ccccccccc}
\\
	&&&& $\circ$ \\
	&&& $\circ$ && $\circ$\\
	&& $\bullet$ && 1 && $\circ$\\
	&$\bullet$ && $\bullet$ && $\bullet$ && $\circ$  \\
	$\bullet$ && $\bullet$ && $\bullet$ && $\bullet$ && $\bullet$ \\
	\\
\end{tabular}
\end{center}
has $\lambda_0 = 4$ and $\lambda_1=1$. All other $\lambda_i$ are $0$. 
It satisfies the conditions $\ref{iBorelFixed1})$ and $\ref{iBorelFixed2})$ of Proposition \ref{pSpaceCurveDiagram}.
Every scheme $C \subset \P^3$ of dimension $1$ with this diagram
has degree $5$ and arithmetic genus $2$. Furthermore it lies on a (not necessarily irreducible) quadric hypersurface.

Notice that diagrams are, by their definition, infinite arrays, but if one horizontal row consists entirely of $\bullet$'s, by Proposition \ref{pSpaceCurveDiagram} \ref{iBorelFixed2}), all subsequent horizontal rows consist of $\bullet$'s, too, so we usually don't depict them.

The generic hyperplane section $\Gamma \subset \P^2$ of  a 
scheme $C \subset \P^3$ as above, has diagram

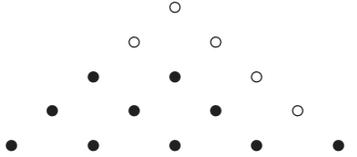
\begin{center}
\begin{tabular}{ccccccccc}
\\
	&&&& $\circ$ \\
	&&& $\circ$ && $\circ$\\
	&& $\bullet$ && $\bullet$ && $\circ$\\
	&$\bullet$ && $\bullet$ && $\bullet$ && $\circ$  \\
	$\bullet$ && $\bullet$ && $\bullet$ && $\bullet$ && $\bullet$ . \\
	\\
\end{tabular}
\end{center}
\end{example}

\begin{remark}
Notice that Proposition \ref{pSpaceCurveDiagram} \ref{iArithmeticGenus})  gives an upper bound for the arithmetic genus if one only considers the arrangement of $\circ$'s, i.e. the diagram of a generic hyperplane section. For example, every $C \subset \P^3$
with arrangement of $\circ$'s as in Example \ref{eDeg5g2} has arithmetic genus at most $3$.
\end{remark}

\begin{remark}
All diagrams as in Proposition \ref{pSpaceCurveDiagram} can be realised by monomial ideals. Thus, in order to get some restriction on the number of possible
diagrams, one has to assume some additional conditions. In our case the condition will be that $C$ lies on an irreducible cubic surface.
\end{remark}

\begin{proposition} \label{pSecantLine}
Let $\Gamma \subset \P^2$ be a finite scheme. If $\lambda_0(\Gamma) > \lambda_1(\Gamma)+2$ then there exists a line
$L\subset \P^2$ such that  $L \cap \Gamma$ contains a
scheme of length $\lambda_0(\Gamma)$.
\end{proposition}

\begin{proof}
If $\lambda_1>0$, this is a special case of a Theorem by Ellia and Peskine \cite[Theorem 4.4]{Green98}. If $\lambda_1=0$,
then $f_\Gamma(1,0)=\bullet$ and all of $\Gamma$ lies on a line. Since $\deg \Gamma = \lambda_0$ in this case, the same conclusion holds.
\end{proof}

\begin{example}
For $\Gamma \subset \P^2$ in the Example \ref{eDeg5g2} we have $\lambda_0=4$, $\lambda_1 = 1$. 
In particular, the conditions of the Proposition \ref{pSecantLine} are satisfied and we have a line $L \subset \P^2$ that contains $4$ of the $5$ points of $\Gamma$ (counted with multiplicity).
\end{example}

%\begin{corollary} \label{cPlaneCurve}
%Let $C \subset \P^2$ a scheme of dimension $1$. If $\lambda_0(C) > \lambda_1(C) + 2$ then $C$ contains
%a plane curve of degree $\lambda_0(C)$. {\color{red} no. It should say: plane curve of degree $\lambda_0$ or higher.}
%{\color{red} DO WE ASSUME THAT $C$ IS CONTAINED IN A SMOOTH SURFACE AWAY FROM FINITELY MANY POINTS??? Otherwise, I think it is false (take two skew lines in $\P^3$, each with multiple structure given by ideals squared in $\P^3$; this should have degree $6$, a generic plane section contains a subscheme of length $4$ that lies on a line, but $C$ contains no planar degree $4$ component whichever way one looks at it.)}
%\end{corollary}

%\begin{proof}
%By Proposition \ref{pSecantLine} every general hyperplane $\P^2 \subset \P^3$ contains a $\lambda_0(C)$-secant line {\color{red} at least a $\lambda_0$ secant, could also be a higher secant...}
%of $C$. This can happen only if $C$ contains a plane curve of degree $\lambda_0(C)$. {\color{red} I think this argument should have more details. At least we are using $\lambda_0 \ge 3$ since otherwise it is just not true}
%\end{proof}

\begin{proof}[Proof of Theorem \ref{tSexticCurves}]
Let $C \subset \P^3$ be a scheme of dimension $1$ and degree $6$. We consider the diagram of a generic hyperplane section 
$\Gamma$ of $C$. By Proposition \ref{pHyperplaneSection} and Propostion \ref{pSpaceCurveDiagram}
its diagram $\Delta(\Gamma)$ will have six $\circ$'s. There are only 4 possible arrangements of  six circles that satisfy the conditions of 
Proposition \ref{pSpaceCurveDiagram}:

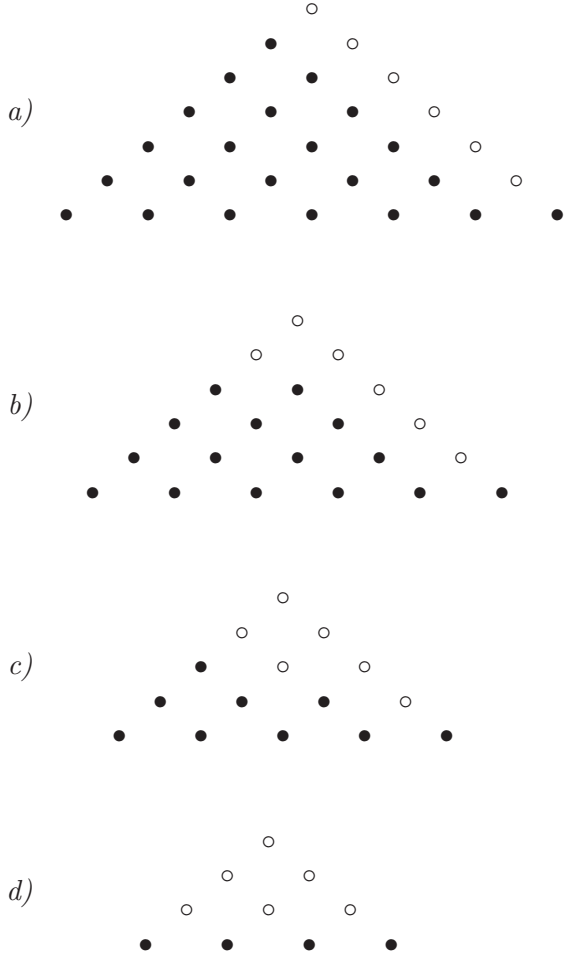
\begin{enumerate}
\item \begin{tabular}{ccccccccccccc}
\\
	&&&&&&$\circ$ \\
	&&&&&$\bullet$ && $\circ$\\
	&&&&$\bullet$&& $\bullet$&& $\circ$\\
	&&&$\bullet$&&$\bullet$&&$\bullet$&& $\circ$  \\
	&&$\bullet$&&$\bullet$&&$\bullet$&&$\bullet$&& $\circ$ \\
	&$\bullet$&&$\bullet$&&$\bullet$&&$\bullet$&&$\bullet$&& $\circ$ \\
	$\bullet$&&$\bullet$&&$\bullet$&&$\bullet$&&$\bullet$&&$\bullet$&& $\bullet$ \\
	\\
\end{tabular}

\item \begin{tabular}{ccccccccccccc}
\\
	&&&&&&$\circ$ \\
	&&&&&$\circ$ && $\circ$\\
	&&&&$\bullet$&& $\bullet$&& $\circ$\\
	&&&$\bullet$&&$\bullet$&&$\bullet$&& $\circ$  \\
	&&$\bullet$&&$\bullet$&&$\bullet$&&$\bullet$&& $\circ$ \\
	&$\bullet$&&$\bullet$&&$\bullet$&&$\bullet$&&$\bullet$&& $\bullet$ \\
	\\
\end{tabular}

\item \begin{tabular}{ccccccccccccc}
\\
	&&&&&&$\circ$ \\
	&&&&&$\circ$ && $\circ$\\
	&&&&$\bullet$&& $\circ$&& $\circ$\\
	&&&$\bullet$&&$\bullet$&&$\bullet$&& $\circ$  \\
	&&$\bullet$&&$\bullet$&&$\bullet$&&$\bullet$&& $\bullet$ \\
	\\
\end{tabular}

\item \begin{tabular}{ccccccccccccc}
\\
	&&&&&&$\circ$ \\
	&&&&&$\circ$ && $\circ$\\
	&&&&$\circ$&& $\circ$&& $\circ$\\
	&&&$\bullet$&&$\bullet$&&$\bullet$&& $\bullet$  \\
	\\
\end{tabular}

\end{enumerate}

Case $a)$ has $\lambda_0=6$ and $\lambda_1=0$. By Proposition \ref{pSecantLine}, a general plane section of $C$ contains a scheme of length $6$ lying on a line $L$. If $C$ also lies on an irreducible cubic surface $S$, each such $L$ has to be contained in $S$ for degree reasons. The family of such lines has dimension at least $2$ since there is a $3$ dimensional family of planes in $\P^3$, and each $L$ as above lies in a pencil of such hyperplanes. But this is a contradiction since no irreducible cubic surface contains a $2$-dimensional family of lines.

Case $b)$ has $\lambda_0=5$ and $\lambda_1=1$. By Proposition \ref{pSecantLine}, a general plane section of $C$ contains a scheme of length $5$ lying on a line $L$. If $C$ also lies on an irreducible cubic surface $S$, each such $L$ has to be contained in $S$ for degree reasons again. This leads to the same contradiction as in the preceding case.

For curves with the arrangement $c)$ we get $p_a \le 0+0+1+1+2=4$ by Proposition \ref{pSpaceCurveDiagram} \ref{iArithmeticGenus}).

For curves with the arrangement $d)$ we get $p_a \le 0+0+1+1+1=3$ by Proposition \ref{pSpaceCurveDiagram} \ref{iArithmeticGenus}).

It follows that $C$ can have arithmetic genus at most $4$ if it lies on an irreducible cubic hypersurface.
\end{proof}

\begin{proposition}\label{pCompleteIntersection}
Every scheme $C \subset \P^3$ of dimension $1$, degree $6$ and arithmetic
genus $4$ that is contained in an irreducible cubic is a complete intersection of type $(2,3)$.
\end{proposition}

\begin{proof}
This also follows from the proof of Theorem \ref{tSexticCurves}. 
Indeed if $p_a=4$ and the sextic curve lies on an irreducible cubic hypersurface, the only possible arrangement
of circles in the diagram of a generic hyperplane section $\Gamma \subset \P^2$ would be $c)$. In this case the arithmetic genus of $C$ can only be equal to $4$ if no numbers occur in the diagram 
of $C$. So $C$ must also have diagram

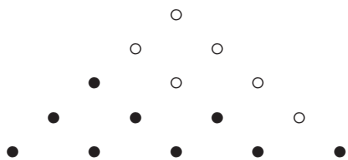
\begin{center}
\begin{tabular}{ccccccccccccc}
\\
	&&&&&&$\circ$ \\
	&&&&&$\circ$ && $\circ$\\
	&&&&$\bullet$&& $\circ$&& $\circ$\\
	&&&$\bullet$&&$\bullet$&&$\bullet$&& $\circ$  \\
	&&$\bullet$&&$\bullet$&&$\bullet$&&$\bullet$&& $\bullet$ \\
	\\
\end{tabular}
\end{center}

\noindent
But then we have $f(2,0) = \bullet$ and therefore $C$ lies on a (not necessarily irreducible) quadric hypersurface. Since
$C$ also lies on an irreducible cubic, it must be a complete intersection $(2,3)$ for degree reasons.
\end{proof}

\end{document}